\documentclass[a4paper,11pt]{scrartcl}

\usepackage[english]{babel}
\usepackage[utf8]{inputenc}
\usepackage{amsmath, amsthm, amssymb} % math stuff
\usepackage{graphicx}
\usepackage[colorinlistoftodos]{todonotes} % color settings
\usepackage{eulervm} % font
\usepackage{wasysym, eurosym} % symbols
\usepackage{textcase, titleps} % upper case text transformation
\usepackage{tabto} % for tabs
\usepackage{tikz,wrapfig} % graphs and stuff
\usetikzlibrary{decorations.pathmorphing} % snake line
\usepackage{enumitem}
\setitemize{itemsep=0cm}

\usepackage[style=numeric, backend=biber, maxnames=2]{biblatex}
\newcommand{\norm}[1]{\|#1\|} % norm
\newcommand{\langel}{\langle} % for typos
\let\phi\varphi % nice phi

\DeclareMathOperator{\fix}{Fix}
\DeclareMathOperator{\sub}{Sub}

\makeatletter
\newcommand*{\defeq}{\mathrel{\rlap{%
                     \raisebox{0.3ex}{$\m@th\cdot$}}%
                     \raisebox{-0.3ex}{$\m@th\cdot$}}%
                     =}
\makeatother

\makeatletter
\newcommand*{\eqdef}{=\mathrel{\rlap{%
                     \raisebox{0.3ex}{$\m@th\cdot$}}%
                     \raisebox{-0.3ex}{$\m@th\cdot$}}%
                     }
\makeatother

\newcommand{\bs}[1]{\boldsymbol{#1}}

\newcommand{\ol}[1]{\overline #1}
\newcommand{\ul}[1]{\underline{\smash #1}}

\definecolor{mycolor}{rgb}{0.9,0.9,0.9}

\newcommand{\meinsubtext}{subtext}  

\newpagestyle{main}{
    \sethead[]%  even-left
            []%   even-center
            []% {\textsc{\MakeTextLowercase\sectiontitle}}%     odd-left
            {}%
            {}%              odd-center
            {}%  odd-right
     \footrule      
     \setfoot[\textsc{\MakeTextLowercase\sectiontitle}]%  even-left
            []%   even-center
            [\thepage{ }]%--- \textsc{len}\Circle] %     even-right
            { \textsc{\meinsubtext}}%     odd-left
            {}%              odd-center
            {\thepage{ }}%--- \textsc{len}\Circle}%  odd-right
  }
  
\swapnumbers
\theoremstyle{plain} % italic font

\newtheorem{theorem}{Theorem}[section]
\newtheorem{lemma}[theorem]{Lemma}
\newtheorem{corollary}[theorem]{Corollary}
\newtheorem{proposition}[theorem]{Proposition}

\theoremstyle{definition} % no italic font

\newtheorem{definition}[theorem]{Definition}

\newtheorem{remark}[theorem]{Remark}
\newtheorem{rem}[theorem]{Remark}
\newtheorem{assumption}[theorem]{Assumption}
\newtheorem{assn}[theorem]{Assumption}

\makeatletter
\g@addto@macro{\thm@space@setup}{\thm@headpunct{$\quad$}}
\makeatother %Lücke nach Satzumgebung

\makeatletter
\renewcommand{\fnum@figure}{Fig. \thefigure}
\makeatother

\newcommand{\mip}{\,:\,}

\title{Doctoral Paper Functionals}
\author{Christoph Tietz}

\begin{document}

\renewcommand{\meinsubtext}{quasi-variational inclusions with bifunctions}

\section*{An Order-Theoretical Multi-Valued Fixed Point Approach to Quasi-Variational Inclusions with Bifunctions}

Christoph Tietz, Technische Universität Dresden
%%%%% content

\begin{abstract}
\noindent \textbf{Abstract.} We present an order-theoretical fixed point theorem for increasing multivalued operators suitable for the method of sub-supersolutions and its application to the following multivalued quasi-variational inclusion: Let $\Omega \subset \mathbb R^N$ be a bounded Lipschitz domain and $W = W_0^{1,p}(\Omega)$. Find $u\in W$ such that for some measurable selection $\eta$ of $f(\cdot,u,u)$ it holds
\begin{equation*}
\langle Eu,w-u\rangle + \int_\Omega \eta(w-u) + K(w,u) - K(u,u) \geq 0\quad\text{for all }w\in W,
\end{equation*}
where $E\colon W \to W^\ast$ is an elliptic Leray-Lions operator of divergence form, $f\colon \Omega \times \mathbb R\times \mathbb R \to \mathcal P(\mathbb R)$ is a multivalued bifunction being upper semicontinuous in the second and decreasing in the third argument, and $K(\cdot,u)$ is a convex functional for each $u\in W$. Under weak assumptions on the data we will prove that there are smallest and greatest solutions between each pair of appropriately defined sub-supersolutions.
\end{abstract}

\section{Introduction}

In order to solve variational inequalities and inclusions in various analytical settings lacking coercivity, the method of sub-supersolutions has been applied successfully, see, e.g., \cite{CaHe11,CaLe15,CaTi18,Le15,Le14} and the recent monograph \cite{CaLe21}. 
This method relies on a combination of sophisticated analytical existence results and truncation-techniques with respect to appropriately defined sub-supersolutions. 
In \cite{Ti19, Ti19b} we proposed a framework for this method in order to make the interplay of analysis and order-theory as lucid as possible. This opens the possibility to replace functionals appearing in variational inequalities and inclusions by bifunctionals, which considerably extends the class of solvable problems while maintaining qualitative properties of the solution set. It is the aim of this paper to refurnish this framework, give an easier proof, and present a sophisticated application. 

To be a little bit more concrete, let $W$ be a partially ordered space and, for each $u\in W$, $A_u\colon W \to \mathbb R\cup \{\infty\}$ a functional. 
Then the variational problem is to find $u\in W$ such that
\begin{equation} 
    A_u(w)-A_u(u) \geq 0 \quad\text{for all }w \in W. \label{eq.vi}
\end{equation}
In order to ensure the existence of a solution $u\in W$ of \eqref{eq.vi}, one clearly needs more information. In particular, in order to apply analytical methods, one presumes that the mapping $u\mapsto A_u$ is continuous in some way. Then, one incorporates a subsolution $\ul u \in W$, which is a solution of the appropriately defined variational inequality
\begin{equation*} 
    \ul A_{\ul u}(w) - \ul A_{\ul u}(\ul u) \geq 0 \quad\text{for all }w \in \ul T_{\ul u} \subset W, \label{eq.vi.sub}
\end{equation*}
and a supersolution $\ol u$, which is similarly defined. At this point, order-theory comes into play. For instance, it is preferable that the supremum of increasing sequences of subsolutions is again a subsolution, and that there is a solution above each subsolution. To obtain those results, sophisticated constructions and proofs are needed. However, in doing so, one realizes that it is often possible to extend the scope of the developed theory by allowing that the functional $A_u$ depends on $u$ not only continuously, but order-continuously. This leads to bifunctionals $A_{u,v} $, where $u \mapsto A_{u,v}$ preserves convergence in some sense, and $v\mapsto A_{u,v}$ preserves order in some sense. Thus, if one can solve variational inequalities of the form \eqref{eq.vi} by means of sub-supersolutions, there is good hope that one can also solve the variational inequality
\[A_{u,u}(w) - A_{u,u}(u) \geq 0\quad\text{for all } w \in W.\]
This procedure extends in a natural way to multivalued quasi-variational inclusions as we will see in this paper. In the second section, we will present the related fixed point theory, where we incorporate ideas from \cite{Li2} to obtain a more general form of the framework presented in \cite{Ti19} and a simpler proof that relies on a multivalued Tarski fixed point theorem and not anymore on an abstract chain generating principle. The third section contains some abstract order-theoretical results about quasi-variational inclusions with bifunctions of the general form
\begin{equation*} 
\text{find }u\in K(u,u)\quad\text{s.t.}\quad a(w)-a(u) \geq 0 \quad\text{for some $a\in A(u,u)$ and all } w \in T(u,u),
\end{equation*}
where the elements of $A(u,u)$ are functionals $a\colon W \to \mathbb R\cup\{\infty\}$ and $K(u,u), T(u,u) \subset W$.
To conclude the paper, we will apply the developed theory to give a concise proof without the need for technical approximation considerations for the existence of solutions of the following multivalued quasi-variational inclusion problem considered in \cite{Le15}: Find $u\in W_0^{1,p}(\Omega)$ such that there is a measurable selection $\eta$ of $f(\cdot,u,u)$  such that
\begin{equation*}
\langle Eu,w -u \rangle + \int_\Omega \eta(w-u) + K(w,u) - K(u,u) \geq 0\quad \text{for all }w\in W^{1,p}(\Omega).
\end{equation*}
There, $W = W_0^{1,p}(\Omega)$ is a Sobolev space over a bounded Lipschitz domain $\Omega\subset \mathbb R^N$,
$E\colon W \to W^\ast$ is an elliptic differential operator of Leray-Lions type, $f\colon \Omega \times \mathbb R\times \mathbb R \to \mathcal P(\mathbb R)$ is a multivalued bifunction,
and $K(\cdot,u)\colon W \to \mathbb R\cup\{\infty\}$ is a convex functional for each $u\in W$.
The key features in this problem are the weak assumptions on $K$ and $f$ such that a wide range of quasi-variational inclusions is covered. For example, $K(\cdot,u)$ allows for the characteristic functional $I_{K(u)}\colon W \to \mathbb R \cup\{\infty\}$ of a convex set $K(u) \subset W$, and $f$ is assumed to be upper semicontinuous in the second and increasing in the third argument, so that the mapping $s\mapsto f(x,s,s)$ is in general neither upper semicontinuous nor lower semicontinuous nor monotone. The lack of continuity allows the solution set to be not compact, but still one can show that there are extremal (i.e. smallest and greatest) solutions between every ordered pair of sub-supersolutions. 

We would like to mention that the method developed here can be used to treat the following even more general problem (cf. \cite{CaLe15}): Find $u\in W^{1,p}(\Omega)$ such that there are measurable selections $\eta$ of $f(\cdot,u,u)$ and $\xi$ of $\hat f(\cdot, \gamma u,\gamma u)$ such that
\begin{equation*}
    \begin{split}
        \langle Eu,w-u \rangle &+ \int_\Omega \eta (w-u) dx + \int_{\partial \Omega} \xi (\gamma w - \gamma u)d\sigma \\
        &+ K(w,u) - K(u,u) + \hat K(\gamma w,\gamma u) - \hat K(\gamma u,\gamma u) \geq 0\quad\text{for all }w\in W^{1,p}(\Omega),
    \end{split}
\end{equation*}
where $\gamma u$ denotes the trace of $u$, $\sigma$ the boundary measure on the boundary $\partial \Omega$ of $\Omega$, and where $\hat f\colon \partial \Omega\times \mathbb R\times \mathbb R \to \mathcal P(\mathbb R)$ and $\hat K\colon L^p(\partial\Omega) \times L^p(\partial\Omega) \to \mathbb R\cup\{\infty\}$ have similar properties as $f$ and $K$.

\section{Abstract Subpoint Method}

Let $V = (V,\leq)$ be a \textbf{poset}, i.e. a set $V$ equipped with a partial order $\leq$. We will equip each subset $W \subset V$ with the induced partial order. Further, for $A,B \in \mathcal P(V)$, the power set of $V$, we write
\[A \leq^\ast B \quad :\Longleftrightarrow\quad \text{for all }a \in A \text{ there is } b\in B \text{ such that } a \leq b.\]
Obviously, $\leq^\ast$ is a preorder on $\mathcal P(V)$ which extends $\leq$ in the sense that $a\leq b$ if and only if $\{a\}\leq^\ast \{b\}$. We use $\leq^\ast$ to extend the notion of increasing functions to the case of multifunctions:

\begin{definition}
Let $V$ be a poset and $S\colon V\to \mathcal P(V)$ a multifunction. Then $S$ is called \textbf{increasing upward} if $v\leq w$ in $V$ implies $S(v) \leq^\ast S(w)$.
\end{definition}

Obviously, a single-valued function $s\colon V \to V$ is increasing if and only if the corresponding multifunction $S\colon V \to \mathcal P(V)$ with $S(v) = \{s(v)\}$ is increasing upward. In the following, we will identify $s$ with $S$ and each $v\in V$ with the singleton $\{v\}$. 

Furthermore, we extend the notion of fixed points to multifunctions as follows:

\begin{definition}
Let $V$ be a poset, $W \subset V$, and $S \colon V \to \mathcal P(W)$ a multifunction. Then the set $\sub S$ of \textbf{subpoints} of $S$ and the set $\fix S$ of \textbf{fixed points} of $S$ are given by
    \[\sub S := \{v \in W\mip v \leq^\ast S(v)\},\qquad \fix S := \{v \in W\mip v \in S(v)\}.\]
\end{definition}

Note that subpoints of $S$ belong by definition to the codomain $W$. This will be important when assigning different properties to the domain $V$ and its subset $W$. In the following, we will first provide results concerning maximal subpoints and fixed points and then concerning greatest ones. Recall that $v^\ast$ is \textbf{maximal} in some poset $V$ if for each $w\in V$ with $v^\ast \leq w$ it follows $v^\ast = w$, and that $v^\ast$ is the \textbf{greatest} element of $V$ if $w\leq v^\ast$ for all $w\in V$. \textbf{Minimal} and \textbf{smallest} elements are defined by duality.

The interplay of increasing upward multifunctions and subpoints as well as fixed points is illustrated in the following proposition:

\begin{proposition}
\label{prop.meta}
Let $V$ be a poset, $W \subset V$, and let $S \colon V \to \mathcal P(W)$ be increasing upward. Then every maximal subpoint of $S$ is a maximal fixed point of $S$. \end{proposition}

\begin{proof}
Let $v^\ast$ be a maximal element of $\sub S$. Then there is $w\in S(v^\ast)$ such that $v^\ast \leq w$. Since $S$ is increasing upward, it follows $S(v^\ast) \leq^\ast S(w)$ and especially $w \leq^\ast S(w)$, i.e. $w\in \sub S$. By maximality of $v^\ast$, we obtain $v^\ast = w$ and thus $v^\ast \in S(v^\ast)$, i.e. $v^\ast \in \fix S$.

Now, let $w\in \fix S$ with $v^\ast \leq w$. Then from $\fix S \subset \sub S$ and the maximality of $v^\ast$ we obtain $v^\ast = w$, i.e. $v^\ast$ is maximal in $\fix S$.
\end{proof}

This result is useful because order-theoretical existence results can be applied with more ease to $\sub S$ than to $\fix S$. In \cite{Ti19,Ti19b}, e.g., we have applied an abstract chain generating principle from \cite{HeLa94}. Here, let us present a simpler approach which is essentially given in \cite{Li2}, where the notion of universally inductive subsets (which we will simply call inductive in their supersets, see below) was introduced. The key idea is to extend the famous Fixed Point Theorem of Tarski (whose proof works via inspection of subpoints) to the multi-valued case.

In order to prove the existence of a maximal subpoint, we will use two notions concerning the existence of special upper bounds:

\begin{definition}
Let $V$ be a poset and $C,D \subset V$.
\begin{enumerate}
    \item[(i)] $d\in D$ is called an \textbf{upper bound} of $C$ if $C \leq^\ast d$.
    \item[(ii)] $C$ is called a \textbf{chain} if it is totally ordered. $D$ is called \textbf{inductive} in $V$, if for each chain $C \subset V$ from $C\leq^\ast D$ it follows $C \leq^\ast d$ for some $d\in D$.
    \item[(iii)] If $C$ has a smallest upper bound in $V$, this upper bound is called \textbf{supremum} of $C$ and denoted by $\sup C$. $V$ is called \textbf{chain-complete} if $\sup C$ exists in $V$ for each chain $C\subset V$.
\end{enumerate}

\end{definition}

As usual, a poset $V$ is called \textbf{inductive} if it is inductive in itself, i.e. if each chain $C\subset V$ is bounded by some $v\in V$. Then, the famous Lemma of Zorn states that each inductive poset has a maximal element $v^\ast$. To ensure the existence of a maximal subpoint it thus suffices to provide conditions under which $\sub S$ is inductive. The following purely order-theoretical result in combination with Proposition \ref{prop.meta} is a slight generalization of \cite[Theorem 3.4]{Li2}:

\begin{proposition} \label{prop.subS.inductive}
Let $V$ be a chain-complete poset, $W \subset V$, and let $S \colon V \to \mathcal P(W)$ be increasing upward and such that all its values $S(v)$ are inductive in $W$. Then $\sub S$ is inductive.
\end{proposition}

\begin{proof}
Let $C \subset \sub S$ be a chain. Then $C \subset V$ and thus $c := \sup C$ exists in $V$. Now, for every $v \in C$ we have
$v \leq^\ast S(v) \leq^\ast S(c)$,
since $S$ is increasing upward. Since $\leq^\ast$ is transitive, we obtain $C \leq^\ast S(c)$, and since $S(c)$ is inductive in $W$, there is some $c^\ast \in S(c)$ such that $C \leq^\ast c^\ast$. By definition of $c$, we have $c \leq c^\ast$ and thus
$c \leq c^\ast \in S(c) \leq^\ast S(c^\ast)$,
which implies that $c^\ast \in \sub S$ is an upper bound of $C$. Consequently, $\sub S$ is inductive.
\end{proof}

\begin{remark}
Since $\emptyset$ is a chain, it follows that every inductive set is non-empty and that each chain-complete poset has a smallest element $\ul u$. So, let $V$ be a poset with smallest element $\ul u$, let $W\subset V$ and let $S\colon V \to \mathcal P(W)$ be increasing upward and such that all its values are non-empty. Take $u^\ast \in S(\ul u)$, then $\ul u \leq u^\ast$ and thus $u^\ast \in S(\ul u) \leq^\ast S(u^\ast)$, meaning that $u^\ast \in \sub S$, i.e. $\sub S \neq \emptyset$.

\end{remark}

There are simple criteria for sets to be chain-complete or inductive if one considers \textbf{ordered topological spaces}, i.e. posets $V$ equipped with a topology $\tau$ such that for all $a, b \in V$ the order intervals 
\begin{equation*}
a^\uparrow := \{c\in V: a\leq c\},\qquad b^\downarrow := \{c\in V: c\leq b\},\qquad [a,b] := a^\uparrow \cap b^\downarrow
\end{equation*} 
are closed. We recall them from \cite{Li2}, but give a much simpler proof for the second part, which is inspired by \cite[Proposition 1.1.4]{HeLa94}.

\begin{proposition}
Let $V$ be an ordered topological space and let $A\subset V$ be compact.
\begin{enumerate}
    \item[(i)] If $A$ is closed and non-empty, then $A$ is inductive in $V$.
    \item[(ii)] If $A$ has a smallest element, then $A$ is chain-complete.
\end{enumerate} 
\end{proposition}

\begin{proof}
Ad (i), let $C\subset V$ be a chain such that $C\leq^\ast A$. If $C=\emptyset$, then any $a\in A$ satisfies $C\leq^\ast a$. If $C\neq \emptyset$, then the family $\mathcal C := \{c^\uparrow\cap A\,:\, c\in C\}$ consists of closed subsets of the compact set $A$ and has the finite intersection property. Thus there is some $a \in \bigcap \mathcal C$, such that $C\leq^\ast a$ and $a\in A$. Consequently, $A$ is inductive in $V$.

Ad (ii), let $C \subset A$ be a chain. If $C = \emptyset$, then $\sup C$ is the smallest element of $A$. If $C \neq \emptyset$, then the closure $\overline C$ of $C$ is non-empty. Then the family $\mathcal C := \{c^\uparrow \cap \overline C\,:\,c\in C\}$ consists of closed subsets of the compact set $A$ and has the finite intersection property. Thus there is some $a \in \bigcap \mathcal C$, such that $C \leq^\ast a$ and $a \in \overline C$. For any upper bound $b$ of $C$ we have $\overline C \subset b^\downarrow$, and so it follows $a \leq b$ and $\sup C = a \in A$. Consequently, $A$ is chain-complete.
\end{proof}

It is straightforward to combine the results so far to obtain the following fixed point theorem, from which \cite[Theorem 2.11]{Ti19} readily follows:

\begin{theorem}\label{thm.abstract.1}
Let $V$ be a poset, $W\subset V$, and let $V$ and $W$ be equipped with topologies such that they become ordered topological spaces. Further, let $D \subset V$ be a compact set with smallest element $\ul u$, and let $S\colon D \to \mathcal P (D \cap W)$ be an increasing upward multifunction whose values are non-empty and compact and closed in $W$. Then $S$ has a maximal fixed point.
\end{theorem}

\begin{remark}
If $V$ is an ordered topological space then $W$ becomes an ordered topological space if equipped with the induced topology as a subset of $V$ (and, as always, with the induced partial order). However, there is no need to use the induced topology on $W$.
\end{remark}

In order to obtain even greatest fixed points, we use the following elementary results from \cite[Section 2.3.2]{Ti19}:

\begin{proposition}\label{prop.greatest}
Let $V$ be an \textbf{upper semilattice} (i.e. a poset in which $\sup\{a,b\}$ exists for all $a,b \in V$) and let $\ul S, S\colon V \to \mathcal P(V)$ be multifunctions.
\begin{enumerate}
    \item[(i)] If $S$ has a maximal fixed point, then $S$ has a greatest fixed point if and only if $\fix S$ is \textbf{directed upward} (i.e. for all $a,b \in \fix S$ there is $c\in \fix S$ such that $\{a,b\} \leq^\ast c$).
    \item[(ii)] If all values of $S$ are directed upward and if $S$ is \textbf{permanent upward} (i.e. from $a \leq b$ it follows $S(a) \subset S(b)$), then $\fix S$ is directed upward.
    \item[(iii)] If $\ul S$ and $S$ are equivalent in the sense that $S(v)\leq^\ast \ul S(v) \leq^\ast S(v)$ for all $v\in V$, then $S$ is increasing upward if and only if $\ul S$ is increasing upward.
    \item[(iv)] If $\ul S$ and $S$ are equivalent in the sense that $\fix S \leq^\ast \fix \ul S \leq^\ast \fix S$, then $S$ and $\ul S$ have the same maximal and greatest fixed points.
\end{enumerate}
\end{proposition}

Now, we have all the ingredients to provide the general framework for the abstract subpoint method. To this end, the so called suboperator $\ul S$ has to satisfy various order-theoretical assumptions, while the main operator $S$ is asked to have good topological properties.

\begin{theorem}\label{thm.greatest.fixedpoint}
Let $V$ be a poset, $W\subset V$, and let $V$ and $W$ be equipped with topologies such that they become ordered topological spaces. Further, let $D \subset V$ and let $\ul S\colon D \to \mathcal P(D)$ and $S\colon D \to \mathcal P(D \cap W)$ be multifunctions such that the following conditions are satisfied:
\begin{enumerate}
    \item[(i)] $D$ is an upper semilattice and compact in $V$,
    \item[(ii)] all values of $S$ are compact and closed in $W$,
    \item[(iii)] $\ul S$ is permanent upward, all its values are directed upward, and for all $v \in D$ it holds $S(v) \subset \ul S(v) \leq^\ast S(v)$,
    \item[(iv)] there is $\ul u \in D$ such that $\ul u \leq^\ast \ul S(\ul u)$.
\end{enumerate}
Then $S$ has a greatest fixed point $u^\ast$ and it holds $\ul u \leq u^\ast$.
\end{theorem}

\begin{proof}
Without loss of generality, we can assume $\ul u \in \sub \ul S$, for otherwise from $\ul u \leq^\ast \ul S(\ul u)\leq^\ast S(\ul u)$ there is some $\ul w$ such that $\ul u \leq \ul w \in S(\ul u) \subset \ul S(\ul u)\leq^\ast \ul S(\ul w)$, implying $\ul w \in \sub \ul S$, so that we can replace $\ul u$ by $\ul w$.

Further, we can assume that $\ul u$ is the smallest element of $D$, for otherwise we restrict our considerations to $\ul u^\uparrow$ and check that the restrictions of $D$, $\ul S$ and $S$ to $\ul u^\uparrow$ satisfy the same properties as $D$, $\ul S$ and $S$ (for this we need $\ul u \in W$ such that $\ul u^\uparrow\cap W$ is closed in $W$). We conclude that the values of $S$ are non-empty, since for all $v\in D$ we have $\ul u \leq v$ and thus $\emptyset \neq \ul S(\ul u) \leq^\ast \ul S(v) \leq^\ast S(v)$.

By Proposition \ref{prop.greatest}(iii) we have that $S$ is increasing upward, and thus Theorem \ref{thm.abstract.1} ensures that $S$ has a maximal fixed point. Especially, by replacing $\ul u$ with any $\ul w \in \sub \ul S$, the arguments so far imply $\fix \ul S \subset \sub \ul S \leq^\ast \fix S$, whereas the inclusion $\fix S \subset \fix \ul S$ follows from $S(v) \subset \ul S(v)$. Thus, by repeated application of Proposition \ref{prop.greatest} we obtain that $S$ has a greatest fixed point. 
\end{proof}

\begin{remark} 
In applications, the fixed points of $S$ will be the solutions of a problem, while the fixed points or even the subpoints of $\ul S$ will be so called subsolutions. Since it is natural to define subsolutions in a way that a solution is always a subsolution, the seemingly strong condition $S(v) \subset \ul S(v)$ is natural. However, it is often non-trivial to provide the counterpart $\ul S(v) \leq^\ast S(v)$ and the directedness of $\ul S(v)$. In Section \ref{sec.application} we will present some analytic tools for this task which, however, depend on the existence of $\ul u \in \fix \ul S$.
\end{remark}

\begin{remark}
All results so far remain true if we replace the partial order $\leq$ in all statements by its dual order $\geq$. Thus, a multifunction $\overline S \colon D \to \mathcal P(D)$ and the existence of some $\overline u \in \fix \overline S$ such that $\underline u \leq \overline u$ can be used to find smallest fixed points of $S$ in $D = [\underline u,\overline u]$. 
\end{remark}

\section{Quasi-Variational Inclusions with Bifunctions}
\label{sec.variational}

Let us consider quasi-variational inclusions with bifunctions, i.e. variational inequalities with multi-valued functions in which the constraint set for the solution, the function to be minimized, as well as the set of test functions depend on the unknown solution itself, and this in a twofold way. To this end, let again $V$ be a poset and assume that $W \subset V$ is a \textbf{lattice} (i.e. for all $u,w\in W$ the supremum $u\vee w:= \sup\{u,v\}$ and the infimum $u\wedge w$ (the largest lower bound of $u$ and $w$) exist in $W$). Further, let $\Gamma(W)$ be the set of functions $a\colon W \to \mathbb R \cup \{\infty\}$ for which $a(u) \in \mathbb R$ for at least one $u\in W$, i.e. the \textbf{effective domain} $\mathcal D(a) := \{u\in W: a(u) \in \mathbb R\}$ is non-empty. Then we consider the following general problem:

\begin{definition}
Let $A\colon W\times V \to \mathcal P(\Gamma(W))$ and $K, T \colon W\times V \to \mathcal P(W)$ be multi-valued bifunctions. Then we say that $u\in W$ is a \textbf{solution} of the associated \textbf{quasi-variational inclusion problem} (QVIP) if and only if 
\begin{equation*} 
u\in K(u,u)\quad\text{and}\quad a(w)-a(u) \geq 0 \quad\text{for some $a\in A(u,u)$ and all } w \in T(u,u)
\end{equation*}
(which implicitly implies $u \in \mathcal D(a)$).
In this context, a function $u$ is admissible if and only if $u\in K(u,u)$, and the set $T(u,u)$ contains the corresponding test functions.
\end{definition}

Naturally, we can interpret (QVIP) as a fixed point problem by fixing one argument of the bifunctions. This leads to the following simpler problem:

%As an abstract prototype of a parameterized variational inclusion we consider the following problem:

\begin{definition}
\label{def.solution_operator}
Let $A\colon W\times V \to \mathcal P(\Gamma(W))$ and $K, T \colon W\times V \to \mathcal P(W)$ be multi-valued bifunctions. Then for each $v\in V$ we say that $u\in W$ is a \textbf{solution} of the associated \textbf{parameterized quasi-variational inclusion problem} (QVIP$_v$) if and only if 
\begin{equation} 
u\in K(u,v)\quad\text{and}\quad a(w)-a(u) \geq 0 \quad\text{for some $a\in A(u,v)$ and all } w \in T(u,v). \label{eq.vipv}
\end{equation}
The associated multi-valued \textbf{solution operator} is defined as
\[S\colon V \to \mathcal P(W),\quad S(v) := \{u \in W \,:\, u \text{ is a solution of (QVIP$_v$)}\}.\]
\end{definition}

Clearly, $u$ is a solution of (QVIP) if and only if $u$ is a solution of (QVIP$_u$), i.e. $u \in \operatorname{Fix}S$. Thus, by introducing subsolutions, we may be able to apply the abstract framework. It will come in useful to define subsolutions as solutions of certain appropriate defined quasi-variational inclusions of the form \eqref{eq.vipv} with slightly modified multifunctions $\underline A$, $\underline K$ and $\underline T$. Thus, we next provide some results about the monotone behavior of the solution operator $S$. 
To this end, we use the following relation for $a,b\in \Gamma(W)$, which goes well with variational problems on lattices:
\[a \preccurlyeq b \quad :\Longleftrightarrow\quad a(u\wedge w) + b(u\vee w) \leq a(u) + b(w)\quad\text{for all }u,w\in W.\]
This relation was introduced in \cite{Le14} as a natural extension of the so called strong set order on sets $A,B \in \mathcal P(W)$, i.e.
\[A \preccurlyeq B \quad :\Longleftrightarrow\quad u\wedge w \in A\quad\text{and}\quad u\vee w \in B\quad\text{for all }u\in A, w\in B\]
(we use the same relation symbol since $A\preccurlyeq B$ holds for non-empty $A,B \in \mathcal P(W)$ if and only if $I_A \preccurlyeq I_B$ holds for their characteristic functionals). Let us recall a few elementary properties: If $W$ contains two non-comparable elements %and if $R$ contains at least two elements,
then $\preccurlyeq$ is not reflexive. Thus, functions $a \in \Gamma(W)$ with $a \preccurlyeq a$ are meaningful; they are commonly called \textbf{submodular} (and a set $A\in\mathcal P(W)$ with $A\preccurlyeq A$ is called \textbf{sublattice} of $W$). Furthermore, $\preccurlyeq$ is in general neither symmetric nor anti-symmetric nor transitive. However, the following modified transitivity holds true:

\begin{proposition}\label{prop.transitive}
Let $a,b,c \in \Gamma(W)$ such that $a\preccurlyeq b$ and $b\preccurlyeq c$. 
\begin{enumerate}
    \item[(i)] Then $\mathcal D(a) \preccurlyeq \mathcal D(b) \preccurlyeq \mathcal D(c)$ and from this it follows $\mathcal D(a) \preccurlyeq \mathcal D(c)$.
    \item[(ii] If $b \preccurlyeq b$ and if $W$ is a \textbf{distributive} lattice (i.e. $\wedge$ and $\vee$ distribute over each other), then $a \preccurlyeq c$.
\end{enumerate}

\end{proposition}

\begin{proof}
Assertion (i) follows readily. The proof of assertion (ii) is a little bit more complicated since $a$, $b$ and $c$ may attain infinite values. Thus, let $u,w \in W$ be arbitrary. We claim that 
\begin{equation} 
a(u\wedge w) + c(u\vee w) \leq a(u) + c(w).\label{eq.reflexive.0}
\end{equation}
If $a(u) = \infty$ or $c(w) = \infty$, then this holds true trivially. Otherwise, let $v\in W$ such that $b(v) \in \mathbb R$. Then from $a\preccurlyeq b$ and $u\wedge(w\wedge(u\vee v)) = u\wedge w$ we obtain
\begin{equation} 
a(u\wedge w) + b(u\vee(w\wedge(u\vee v))) \leq a(u) + b(w\wedge(u\vee v)).\label{eq.reflexiv.1}
\end{equation}
By dual reasoning, we also obtain
\begin{equation} 
c(w\vee u) + b(w\wedge(u\vee(w\wedge v))) \leq c(w) + b(u\vee(w\wedge v)).\label{eq.reflexiv.2}
\end{equation}
All terms in \eqref{eq.reflexiv.1} and \eqref{eq.reflexiv.2} are finite (recall that $\mathcal D(a) \preccurlyeq \mathcal D(b) \preccurlyeq \mathcal D(c)$) and thus \eqref{eq.reflexive.0} follows from $b\preccurlyeq b$ and the identity 
\begin{align*}
[u\vee(w\wedge(u\vee v))] \wedge [w \wedge (u\vee(w\wedge v))] 
&= (u\vee w) \wedge (u\vee v) \wedge w \wedge (u\vee w) \wedge (u\vee v)\\
&=w\wedge(u\vee v)
\end{align*}
together with its dual counterpart.
\end{proof}

Despite $\preccurlyeq$ being no partial order on $\Gamma(W)$, we use the same terminology as for partial orders. Especially, $a^\uparrow = \{b\in \Gamma(W): a\preccurlyeq b\}$ and we have the following  notion:

\begin{definition} \label{def.upperincreasing}
A multifunction $A\colon V \to \mathcal P(\Gamma(W))$ is called \textbf{increasing upward} if $v \leq v'$ in $V$ implies $A(v) \preccurlyeq^\ast A(v')$, i.e. for all $a\in A(v)$ there is $b\in A(v')$ such that 
\[a(u\wedge w) + b(u\vee w) \leq a(u) + b(w)\quad\text{for all }u,w\in W.\]
\end{definition}

\begin{remark}
Note that parameterized problems like (QVIP$_v$) where considered for a long time in the case where $A$ is single-valued and constant, see, e.g., \cite[Lemma 2.8.1]{To98}. There, parameterized functions $a_v \colon W\to \mathbb R$ with the following monotone compatibility property were considered:
\[a_{v}(u\wedge w) + a_{v'}(u\vee w) \leq a_{v}(u) + a_{v'}(w) \quad\text{for all }u,w\in W\text{ and for $v \leq v'$ in $V$}.\]
Obviously, this is a special case of Definition \ref{def.upperincreasing} and consequently, our next lemma generalizes the result of \cite{To98}.
\end{remark}

Now, let us provide the monotonic dependence of the solution operator on the data, which gives us a guideline for the construction of the multifunctions $A$, $K$ and $T$. To this end, we consider a second quasi-variational inclusion problem (QVIP') with data $A'$, $K'$ and $T'$ and solution operator $S'$. 

\begin{lemma}\label{lem.dependence}
Let for all $u\in W$ and $v,v'\in V$ such that $u\in S(v)$ and $v\leq v'$ the following relations hold true:
\[A(u,v) \preccurlyeq^\ast A'(u,v'),\quad K(u,v) \subset K'(u,v'),\quad T'(u,v') \subset T(u,v), \quad T'(u,v') \leq^\ast u. \]
Then $S(v) \subset S'(v')$. %$S$ is permanent upward.
\end{lemma}

\begin{proof}
Let $v \leq v' \in V$ and $u \in S(v)$. Then it holds $u\in K(u,v) \subset K'(u,v')$ and there is $a\in A(u,v)$ such that
\begin{equation}
    a(w)-a(u) \geq 0\quad\text{for all } w\in T(u,v). \label{eq.S.perm.0}
\end{equation}
Since $A(u,v) \preccurlyeq^\ast A'(u,v')$, there is $b\in A'(u,v')$ such that $a\preccurlyeq b$, that is
\[a(u\wedge w) + b(u\vee w) \leq a(u) + b(w)\quad\text{for all } w\in W.\]
Taking $w\in T'(u,v') \subset T(u,v)$, we have $w\leq u$ and thus $u\wedge w = w$ and $u\vee w = u$, resulting in
\begin{equation}
    a(w) + b(u) \leq a(u) + b(w). \label{eq.S.perm.1}
\end{equation}
Using \eqref{eq.S.perm.0} and \eqref{eq.S.perm.1}, we conclude
\begin{equation*}
    b(u) \leq b(u) + a(w) - a(u) \leq a(u) + b(w) - a(u) = b(w),
\end{equation*}
and since $b(w) < \infty$ for at least one $w\in W$, we obtain $u\in\mathcal D(b)$ and $u\in S'(v')$.

\end{proof}

As an immediate consequence of Lemma \ref{lem.dependence} for the case (QVIP) $=$ (QVIP') we obtain conditions on $A$, $K$ and $T$ under which $S$ is permanent upward; especially, $v\mapsto A(u,v)$ should be asked to be increasing upward. Furthermore, one can utilize the fact that the sum of two increasing upward multifunctions is again increasing upward, and there is even a connection to positive linear functionals. To this end, we have to assume that $W$ is an \textbf{ordered linear space} (i.e. $W$ is a linear space over $\mathbb R$ such that $u\leq w$ is equivalent to $0\leq w-u$, and $0\leq \lambda$ in $\mathbb R$ and $0\leq u \in W$ imply $0\leq\lambda w$). Then for $u \in W$ we define $u^+ := u\wedge 0$ and obtain the identities
\begin{equation} 
u - (u\wedge w) = (u-w)^+ = (u\vee w) - w. \label{eq.identity}
\end{equation}
By $W'$ we denote the linear space of all linear functionals $a\colon W \to \mathbb R$, and for $a\in W'$ and $u\in W$ we write $\langle a,u\rangle := a(u)$ and say that $a$ is \textbf{positive} if $\langle a,u^+\rangle \geq 0$ for all $u\in W$.

\begin{proposition}\label{prop.preccurly.positive}
For $a,b \in W'$ it holds $a\preccurlyeq b$ if and only if $a-b$ is positive.
\end{proposition}

\begin{proof}
Let $u,w \in W$. Then from the identities in \eqref{eq.identity} and the linearity of $a$ and $b$ we obtain
\begin{align*}
    a(u\wedge w) + b(u\vee w) \leq a(u) + b(w)  
    \quad & \Longleftrightarrow\quad 
    \langle b, (u\vee w) - w\rangle \leq \langle a,u - (u\wedge w)\rangle \\
    \quad &\Longleftrightarrow\quad 
    \langle a-b, (u-w)^+ \rangle \geq 0.
\end{align*}
The claim follows readily.%Since $(u-w)^+ \geq 0$, the last inequality follows from $a-b \geq 0$, and thus we conclude $a \preccurlyeq b$.
\end{proof}

The calculations above suggest the following notions:
\begin{itemize}
    \item[(i)] For $a, b \in \Gamma(W)$ and $u,w\in W$ we say that $a\preccurlyeq b$ w.r.t. $(u,w)$ if
    \[a(u\wedge w) + b(u\vee w) \leq a(u) + b(w).\]
    \item[(ii)] A multifunction $A\colon W \to \mathcal P(W')$ is called \textbf{T-monotone} if for all $u,w\in W$ and all $a\in A(u)$ and $b\in A(w)$ it holds 
    \[\langle a-b, (u-w)^+\rangle \geq 0.\]
\end{itemize}

\begin{corollary}
Let $A\colon W \to \mathcal P(W')$ be T-monotone and $a \in A(u)$, $b \in A(w)$. Then $a \preccurlyeq b$ w.r.t. $(u,w)$.
\end{corollary}

\begin{remark}
For the use of T-monotone functions $A\colon W \to W'$ see, e.g., \cite[Section 4:5]{Ro87}. The use of T-monotone multifunctions in the context of variational inclusions is standard, see, e.g., \cite{Le15}. %We will present a concrete example in Section \ref{} below.
\end{remark}

\begin{remark}
We have seen that our definition of increasing upward mappings is productive. In the special case $v = v'$ we obtain the following notion: 
A set $A\subset \Gamma(W)$ is called {upper submodular} if $A \preccurlyeq^\ast A$, that is, if for all $a\in A$ there is $b\in B$ such that $a(w\wedge w') + b(w\vee w') \leq a(w) + b(w')$ for all $w,w'\in W$.

\end{remark}

\section{Elliptic Quasi-Variational Inclusions with Bifunctions}
\label{sec.application}

Finally, let us apply the abstract results in order to find solutions of an involved elliptic quasi-variational inclusion problem whose lower order terms are given by bifunctions. To this end, we will presume some knowledge of functional analysis (for more information, we refer to, e.g., \cite{CaLe21,Ze90}) and use the following setting throughout this section:

Let $\Omega\subset \mathbb R^N$, $N\geq 1$, be a bounded domain with Lipschitz boundary. We equip $\Omega$ with the Lebesgue measure and denote by $L^0(\Omega)$ the linear space of all (equivalence classes of) measurable functions $u\colon \Omega \to \mathbb R$ (where $\mathbb R$ is equipped with the Borel measure). We fix $p\in (1,\infty)$ and $q\in (1,p^\ast)$, where $p^\ast$ denotes the {critical Sobolev exponent} associated with $p$ and the dimension $N$, given by $p^\ast = Np/(N-p)$ for $p<N$, and $p^\ast$ arbitrarily large for $p\geq N$. Furthermore, $q'$ denotes the {H\"older conjugate} of $q$. In dependence of these constants, we introduce the following subspaces of $L^0(\Omega)$:
\begin{equation*}
L := L^{q'}(\Omega),\qquad V := L^{p^\ast}(\Omega),\qquad W := W^{1,p}_0(\Omega),
\end{equation*}
where $L^r(\Omega)$ is the usual Lebesgue space of $r$-th power integrable functions, and $W^{1,p}_0(\Omega)$ the usual Sobolev space with zero boundary values in the sense of traces. We will denote the norm of $L^r(\Omega)$ by $\norm{\cdot}_r$, and we note that $\norm{\nabla u}_p$ gives an equivalent norm on $W$ (where $\nabla u$ is the weak gradient of $u$). The spaces are chosen in such a way that the embedding $W\hookrightarrow V$ is continuous and that the embedding $i_q: W \hookrightarrow L^\ast$ and its adjoint $i_q^\ast\colon L \hookrightarrow W^\ast$ are compact (with respect to the norm topology). There, $L^\ast$ and $W^\ast$ denote the topological dual of $L$ and $W$, respectively, and $i_q^\ast$ is the natural embedding, i.e. for $\eta \in L$ we have
\[\langle i_q^\ast\eta, u\rangle := \int_\Omega \eta u := \int_\Omega \eta(x) u(x)\,\mathrm{d}x\quad\text{for all }u\in W.\]
N ote that we suppress arguments in the notation of integrals.

Concerning order, we equip $\mathbb R$ with the usual partial order $\leq$ and $L^0(\Omega)$ (or even the space of (equivalence classes of) functions $u\colon \Omega \to \mathbb R$) and all its subspaces with the natural partial order $\leq$, defined by $u\leq v$ if $u(x)\leq v(x)$ for a.e. $x\in\Omega$, by which $L^0(\Omega)$ is known to be an  ordered linear space with lattice structure and distributive sub-lattices $L$, $V$ and $W$. 

Concerning topology, besides the norm topology we will often use the weak topology on the reflexive Banach spaces $L$, $V$ and $W$. If a sequence $(u_n)$ converges weakly to some function $u$, we write $u_n\rightharpoonup u$. Recall that each interval $u^\uparrow$ is closed and convex and thus weakly closed, such that especially $V$ and $W$ are ordered topological spaces with respect to the weak topology. Recall further, that a subset $K$ of a reflexive Banach space is weakly compact and weakly closed if and only if it is norm bounded and sequentially closed (via Banach-Alaoglu and Eberlein-\v Smulian), such that (by monotonicity of the integral) each order interval $[\ul u,\ol u]$ in $V$ is weakly compact and weakly closed. Note, however, that such intervals in $W$ are not bounded and thus not weakly compact since our partial order is not suitable to control the derivatives of a Sobolev function.

In order to formulate the (QVIP) of this section, let $E\colon W \to W^\ast$ be an elliptic differential operator of Leray-Lions type (for a precise description see Assumption \ref{ass.E} below), $f\colon \Omega \times \mathbb R \times \mathbb R \to \mathcal P(\mathbb R)$ a multi-valued bifunction, and $K\colon W\times V\to \mathbb R\cup\{\infty\}$ a bifunctional. Then the (QVIP) reads as follows:
\begin{equation}
\begin{split}
\label{eq.problem}
\text{Find }u\in W\text{ }&\text{s.t. there is } \eta \in L \text{ with }\eta(x) \in f(x,u(x),u(x))\text{ for a.e. $x\in\Omega$}\\
&\text{s.t. } \langle Eu,w -u \rangle + \int_\Omega \eta(w-u) + K(w,u) - K(u,u) \geq 0\quad \text{for all }w\in W.
\end{split}
\end{equation}

\begin{remark}
This problem was considered in \cite{Le15}, where a rather involved proof of the existence result was given. By using the abstract framework developed above we are able to reduce the analytic work to be done to a minimum. Nevertheless, not all difficulties can be skipped. For the sake of completeness, we are going to present full arguments in our notation although most of the ideas are given in \cite{CaLe15,CaTi18,Le15,Le14,Ti19} and the papers cited there.
\end{remark}

To shorten the notation and to bring this problem in the form (QVIP) presented in Section \ref{sec.variational}, we introduce the mapping
\[A\colon W\times V \to \mathcal P(\Gamma(W)),\quad A(u,v) := E(u) + i^\ast_q F(u,v) + K(\cdot,v),\]
where $E$ is the elliptic differential operator from above and
\[F\colon L^0(\Omega)\times L^0(\Omega) \to \mathcal P(L),\quad F(u,v) := \{\eta \in L: \eta \subset f(\cdot,u,v)\},\]
where $\eta \subset f(\cdot,u,v)$ means $\eta(x) \in f(x,u(x),v(x)$ for a.e. $x\in\Omega$. Then $u\in W$ is a \textbf{solution} of \eqref{eq.problem} if and only if
\[a(w) - a(u) \geq 0\quad\text{for some }a\in A(u,u) \text{ and all }w\in W.\]
There, implicitly, only functions $u\in W$ with $u \in \mathcal D(K(\cdot,v))$ are admissible, and every function $w\in W$ is a test function.

In order to define subsolutions, we modify this setting by expanding $A(u,u)$ and shrinking the space of test functions so that it is as easy as possible for subsolutions to exist. In accordance with \cite{Le15}, we define
\[\underline A\colon W\times V \to \mathcal P(\Gamma(W)),\quad \ul A(u,v) := E(u) + i^\ast_q F(u,v) + K(\cdot,v)^\downarrow\]
(where $K(\cdot,v)^\downarrow = \{k\in\Gamma(W): k \preccurlyeq K(\cdot,v)\}$) and
\[\ul T\colon W\times V \to \mathcal P(W),\quad \ul T(u,v) := u \wedge \mathcal D(K_v).\]
Then, $\ul u\in W$ is called \textbf{subsolution} of \eqref{eq.problem} if and only if
\[a(w) - a(\ul u) \geq 0\quad\text{for some }a\in \ul A(\ul u,\ul u) \text{ and all }w\in \ul T(\ul u,\ul u).\]
Dually, a \textbf{supersolution} $\ol u \in W$ of \eqref{eq.problem} is defined as solution of the quasi-variational inclusion with $K_v^\uparrow$ instead of $K_v^\downarrow$ and $\ol T(u,v) := u \vee \mathcal D(K_v)$ instead of $\ul T(u,v)$.

Next, let us give precise assumptions on the data which guarantee that all conditions of Theorem \ref{thm.greatest.fixedpoint} are fulfilled, so that Problem \eqref{eq.problem} has greatest and smallest solutions between each ordered pair of sub-supersolutions.

\begin{assn}
Since we do not provide sub-supersolutions, we have to assume:
\begin{itemize}
\itemsep0pt
\item[(S)] There is a subsolution $\ul u$ and a supersolution $\ol u$ such that $\ul u \leq \ol u$. %n ordered pair $(\underline u,\overline u)$ of sub-supersolutions. %of problem ($P$)
\end{itemize}
\end{assn}

\begin{remark}
See, e.g., \cite{Le15} for the construction of sub-supersolutions under appropriate conditions by use of the monotone dependence provided in Lemma \ref{lem.dependence}. The key idea is to use solutions of simpler variational inclusions.
\end{remark}

In the following, we will often restrict our consideration to the interval $[\ul u,\ol u]$. Although $\ul u,\ol u \in W$ by definition, this interval of sub-supersolutions is always to be considered in $V$ such that it is non-empty, bounded, closed and weakly compact. Especially, if we merely assume $v\in[\ul u,\ol u]$, then $v$ may not have weak derivatives.

\begin{assn}
%\subsection{Assumptions on $\bs{A}$}
\label{ass.E}
Let $e\colon \Omega \times \mathbb R^N \to \mathbb R^N$ be a function defining the (single-valued) differential operator $E\colon W\to W^\ast$ of Leray-Lions type via 
\[Eu := -\operatorname{div} e(\cdot,\nabla u)\quad\text{resp.}\quad \langel Eu,w\rangle := \int_\Omega e(\cdot,\nabla u)\nabla w\quad\text{for all }w\in W.\]
The following standard assumptions on $e$ are meant to hold for a.e. $x\in\Omega$ and all $\xi\in\mathbb R^N$.

\begin{itemize}
\itemsep0pt
\item[(E1)] The function $e$ is a \textbf{Carath\'eodory function}, i.e. $x\mapsto e(x,\xi)$ is measurable on $\Omega$ and $\xi \mapsto e(x,\xi)$ is continuous on $\mathbb R^N$.
\item[(E2)] The function $e$ is $\boldsymbol{p}$\textbf{-coercive} in the second argument, i.e. there is a constant $\alpha_2 > 0$ and a function $e_2\in L^1(\Omega)$ such that $e(x,\xi) \xi \geq \alpha_2 |\xi|^p - e_2(x)$.
\item[(E3)] There exists a constant $\alpha_3 > 0$ and a function $e_3 \in L^{p'}(\Omega)$ such that $e$ satisfies the \textbf{growth} condition $|e(x,\xi)| \leq \alpha_3 |\xi|^{p-1} + e_3(x)$.
\item[(E4)] The function $e$ is \textbf{monotone} in the second argument, i.e. it holds, for all $\xi,\xi'\in \mathbb R^N$, $(e(x,\xi) - e(x,\xi'))(\xi-\xi') \geq 0$. 
\end{itemize}

\end{assn}

Under conditions (E1)---(E4) the operator $E\colon W\to W^\ast$ is known to be well-defined, bounded, coercive, monotone and continuous, and thus \textbf{pseudomonotone} (i.e. $u_n\rightharpoonup u$ in $W$ and $\limsup_n \langle Eu_n,u_n-u\rangle \leq 0$ imply $\langle Eu,w-u\rangle \geq \limsup_n \langle Eu_n,w-u_n\rangle$ for all $w\in W$). An example is given by the $p$-Laplacian defined by $e(x,\xi) = |\xi|^{p-2}\xi$. See, e.g., \cite[Section 2.4]{Roubicek} for an exhaustive treatment.

\begin{assn}

Let $f\colon \Omega\times\mathbb R\times \mathbb R \to \mathcal P(\mathbb R)$ be a multifunction whose values are compact intervals. If not specified otherwise, the following conditions are meant to hold for a.e. $x\in\Omega$ and all $s,t\in\mathbb R$.

\begin{itemize}
\itemsep0pt
\item[(F1)] The function $(x,t)\mapsto f(x,s,t)$ is \textbf{superpositionally measurable}, i.e. for all $v\in L^0(\Omega)$ the multivalued function $x\mapsto f(x,s,v(x))$ is \textbf{measurable}, i.e. the set $\{x\in\Omega: f(x,s,v(x)) \subset U\}$ is measurable for each open $U\subset\mathbb R$.

\item[(F2)] The function $s\mapsto f(x,s,t)$ is \textbf{upper semicontinuous} on $\mathbb R$, i.e. the set $\{s\in\mathbb R: f(x,s,t) \subset U\}$ is open for each open $U\subset\mathbb R$.
%, i.e., for each $U\subset \mathbb R$ the inclusion $f(x,s,t)\subset U$ implies $f(x,s',t)\subset U$ for all $s'$ near $s$.

\item[(F3)] The function $t\mapsto f(x,s,t)$ is \textbf{decreasing} on $\mathbb R$, i.e. $f(x,s,t)\leq^\ast f(x,s,t')$ and $f(x,s,t') \geq^\ast f(x,s,t)$ for all $t' \in \mathbb R$ with $t'\leq t$.

\item[(F4)] There is some $b_4\in L$ such that $f$ satisfies the \textbf{growth} condition
\begin{equation*}
|y| \leq b_4(x)\quad\text{for all  $y\in f(x,s,t)$ and all $s,t\in [\underline u(x),\overline u(x)]$.}
%\sup\{|y|:y\in f(x,s,t), s,t\in [\underline u(x),\overline u(x)] \} \leq b_4(x).
\end{equation*}

\end{itemize}

\end{assn}

\begin{rem}
Since $f$ may not depend continuously on $t$, it may be a tedious task to ensure (F1). However, if one only assumes (F4) and that $f$ is measurable in $x$ (which follows at once from (F1) and is a standard assumption if $f$ does not depend on $t$), then there may be a fairly simple criterion which ensures (F1).
\end{rem}

From (F4) it follows that all measurable selections $\eta \subset f(\cdot,u,v)$ with $u,v \in [\ul u,\ol u]$ belong to $L$, which means that $F(u,v)$ is non-empty if $f$ is \textbf{weakly superpositionally measurable}, i.e. if $f(\cdot,u,v)$ has at least one measurable selection whenever $u,v\in L^0(\Omega)$. However, this follows readily from (F1) and (F2) and \cite[Lemma 7.1]{Apea95}. Thus, $F$ has non-empty values if restricted to $[\ul u,\ol u]$.

Further, we will need the following weak closure property of $f$, which follows at once from (F1) and \cite[Proposition 2]{CaTi18}:

\begin{proposition}\label{prop.F.closed}
Let $v\in L^0(\Omega)$, $\eta_n\rightharpoonup \eta$ in $L^1(\Omega)$ and $u_n\to u$ in $L^0(\Omega)$ pointwise a.e. If $\eta_n \subset f(\cdot,u_n,v)$, then $\eta \subset f(\cdot,u,v)$.
\end{proposition}

(Note that the convergence in Proposition \ref{prop.F.closed} holds true for at least a subsequence if $\eta_n \rightharpoonup \eta$ in $L^r(\Omega)$ and $u_n \to u$ in $L^s(\Omega)$ for some $r,s \geq 1$ since $\Omega$ is bounded.)

\begin{assumption}

Let $K\colon W \times V\to \mathbb R\cup\{\infty\}$ be a bifunctional satisfying the following conditions:
\begin{itemize}
\item[(K1)] For all $v \in [\ul u,\ol u]$ %\in [\underline u,\overline u]_Y$, 
the function $K_v := K(\cdot,v)\colon W \to \mathbb R\cup\{\infty\}$ is proper (i.e. $\mathcal D(K_v)\neq \emptyset$, thus $K_v\in\Gamma$), convex and lower semicontinuous.%, and satisfies $\mathcal D(K_v) \subset Y$.
\item[(K2)] The mapping $v\mapsto K_v$ is {increasing}, i.e. $v\leq v'$ in $[\ul u,\ol u]$ implies $K_{v} \preccurlyeq  K_{v'}$.%, i.e.
%\begin{equation*}
%J_u(w_1\wedge w_2) + J_v(w_1\vee w_2) \leq J_u(w_1) + J_v(w_2)\quad\text{for all }w_1,w_2 \in [\underline u,\overline u]_K.
%\end{equation*}
\item[(K3)] For each $v\in [\ul u,\ol u]$ there is some constant $c_3>0$ such that for all $u\in [\ul u,\ol u]$ % [\underline u,\overline u]_Y$
it holds $K_v(u) \geq -c_3(\|\nabla u\|_p^{p-1}+1)$.%one has
\end{itemize}
\end{assumption}

It is well-known that (K1) implies that $K_v$ is weakly sequentially lower semicontinuous and that $\partial K_v$ (the subdifferential of $K_v$) is maximal monotone.
As an example, we have $K_v = I_{K(v)}$, the characteristic functional of a non-empty closed convex set $K(v)\subset W$. In this case, $\mathcal D(K_v) = K(v)$ and (K2) becomes $K({v}) \preccurlyeq K({v'})$ for $v \leq v'$, while (K3) holds trivially.

\begin{remark}
Note that we slightly changed the assumptions on the data given in \cite{Le15}. 
In particular, we do not assume that $(x,s) \mapsto f(x,s,t)$ is superpositionally measurable, %or that $f$ is graph-measurable as in \cite{Le15}, 
but only (by (F1) and (F2)) that $(x,s) \mapsto f(x,s,t)$ is weakly superpositionally measurable. Furthermore, our assumption (K3) replaces the technical assumptions (J3) and (J4) of \cite{Le15} which is possible by separating topological and order-theoretical properties. Furthermore, in order to present the main ideas, we restrict our considerations to the space $W^{1,p}_0(\Omega)$. At the cost of some slightly more involved calculation we could treat also quasi-variational inclusions with non-zero boundary conditions.
\end{remark}

We are now in a position to check if all the conditions of Theorem \ref{thm.greatest.fixedpoint} are fulfilled with respect to $D:= [\ul u,\ol u]$, $V$ and $W$ as given in this section, $S$ being the solution operator associated with $A$, $K \equiv W$ and $T \equiv W$, and $\ul S$ being the solution operator associated with $\underline A$, $\ul K \equiv W$ and $\underline T$ as described in Definition \ref{def.solution_operator}. To be precise, in what follows we restrict both the domain and the values of $S$ and $\ul S$ to $D$, since we are only interested in solutions between $\ul u$ and $\ol u$.

First of all, $D$, $V$ and $W$ satisfy all required conditions, and we have not only $\ul u \leq^\ast \ul S(\ul u)$, but even $\ul u \in \fix \ul S$ (which will be relevant in the proof of Theorem \ref{theorem.aux.existence}). It thus remains to provide the topological properties of $S$ and the order-theoretical properties of $\ul S$ and its connection to $S$. Let us start with the properties of $S$ which are proved in a standard way as if we would have no bifunctions:

\begin{proposition}
All values of $S\colon D \to \mathcal P(D\cap W)$ are weakly compact (and thus weakly closed) in $W$.
\end{proposition}

\begin{proof}
Let $v \in D$ be fixed. We only have to prove that $S(v)$ is bounded and weakly sequentially closed. So, let $u\in S(v)$ be arbitrary and fix some $w \in \mathcal D(K_v)$ (not depending on $u$). Then it holds, for some $\eta \in F(u,v)$,
\begin{equation}
\label{eq.estimate.solution}
\langle Eu,w\rangle + \int_\Omega \eta(w-u) + K_v(w) - K_v(u) \geq \langle Eu,u\rangle.
\end{equation}
By (K3) we have $-K_v(u) \leq c_3(\|\nabla u\|_p^{p-1} + 1)$ and we can assume that also $K_v(w) \leq c_3$. Combining these estimates with those from (A2)---(A3) and using H\"older's inequality, we obtain from \eqref{eq.estimate.solution}
\begin{equation}
\label{eq.estimate.solution.2}
\begin{split}
(\alpha_3 \norm{\nabla u}_{p}^{p-1} + \norm{a_3}_{p'})\norm{\nabla w}_{p} 
~+~ \norm{\eta}_{{q'}}(\norm{w}_{{q}} + \norm{u}_{{q}}) 
~&+~ c_3(\|\nabla u\|_p^{p-1} + 2 )\\
&\geq \alpha_2 \norm{\nabla u}_p^p - \norm{e_2}_1.
\end{split}
\end{equation}
Now, $\|\nabla w\|_{p}$ and $\|w\|_q$ are constant, $\|\eta\|_{q'}$ is bounded due to (F4), and we have $\norm{u}_q\leq \norm{\underline u}_q + \norm{\ol u}_{{q}}$, since $u\in D$. Thus, \eqref{eq.estimate.solution.2} can not hold if $\norm{\nabla u}_p$ is arbitrarily large, whence $S(v)$ is bounded in $W$.

It remains to prove that $S(v)$ is weakly sequentially closed. To this end, let $(u_n) \subset S(v)$ and assume $u_n \rightharpoonup u$ for some $u \in W$. From the continuous embedding $W \hookrightarrow V$ we obtain $u_n \rightharpoonup u$ in $V$ (and thus $u\in D$) and from the compact embedding $W \hookrightarrow L^\ast$ we obtain $u_n \to u$ in $L^\ast$. Further, since $K_v$ is weakly sequentially lower semicontinuous due to (K1), from $(u_n)\subset \mathcal D(K_v)$ we have $u\in \mathcal D(K_v)$.
Now, let $(\eta_n) \subset L$ be a sequence such that $\eta_n \in F(u_n,v)$ and
\begin{equation}
\label{eq.weaklysolution.1}
\langle Au_n,w-u_n\rangle + \int_\Omega \eta_n (w-u_n) + K_v(w) - K_v(u_n) \geq 0 \quad\text{for all }w\in W.
\end{equation}
We can assume $\eta_n \rightharpoonup \eta$ in $L$ for some $\eta \in V$ since $(\eta_n)$ is bounded in $L$ due to (F4), and thanks to Proposition \ref{prop.F.closed} we have $\eta\in F(u,v)$.
Letting $w=u$ in \eqref{eq.weaklysolution.1} and passing to the limit yields, using again that $K_v$ is weakly sequentially lower semicontinuous, 
\begin{equation*}
{\limsup}_n \langle Au_n,u_n-u\rangle \leq 0.
\end{equation*}
Since $A$ is pseudomonotone, we infer $\langle Au,w-u\rangle \geq \limsup_n \langle Au_n,w-u_n\rangle$ for all $w\in W$. Using again (\ref{eq.weaklysolution.1}), it follows, for all $w\in W$,
\begin{equation*}
%\label{eq.weaklysolution.5}
\begin{split}
\langle Au,w-u\rangle &+ \int_\Omega\eta (w-u) + K_v(w) - K_v(u) \\
&\geq {\limsup}_n \langle Au_n, w-u_n\rangle + {\lim}_n \int_\Omega \eta_n (w-u_n) + K_v(w) - {\liminf}_n K_v(u_n) \\
&\geq {\limsup}_n\left(\langle Au_n, w-u_n\rangle + \int_\Omega \eta_n (w-u_n) + K_v(w) - K_v(u_n) \right) \geq 0,
\end{split}
\end{equation*}
which proves $u\in S(v)$, thus $S(v)$ is weakly sequentially closed.
\end{proof}

Next, let us check that $\ul S$ is permanent upward. This is almost a purely order-theoretical proof, but note that it is vital for the values of $F$ to be non-empty.

\begin{proposition}
The operator $\ul S\colon D \to \mathcal P(D\cap W)$ is permanent upward.
\end{proposition}

\begin{proof}
We apply Lemma \ref{lem.dependence}. So, let $v\leq v'$ in $D$ and $u\in S(v)$.

\textit{First}, we have to prove $\ul A(u,v) \preccurlyeq^\ast \ul A(u,v')$, where $\ul A(u,v) = E(u) + i_q^\ast F(u,v) + K_v^\downarrow$. We consider the summands separately: 
\begin{itemize}
    \item[(i)] Since $E(u)$ does not depend on $v$, there is nothing to do. 
    \item[(ii)] Let $\eta \in F(u,v)$, let $\eta' \in F(u,v')$ be arbitrary and set $\eta'' := \eta \wedge \eta' \in L$. Further, let $\alpha \colon \Omega \to \mathbb R$ be such that $\alpha \subset f(\cdot,u,v')$ and $\alpha \leq \eta$ (which exists due to (F3)). $\alpha$ may not be measurable, but we have $\alpha \wedge \eta' \leq \eta'' \leq \eta'$. Now, recall that all values of $f$ are closed. It thus follows $\alpha\wedge \eta' \subset f(\cdot,u,v')$ and $\eta'' \subset f(\cdot,u,v')$, i.e. $\eta '' \in F(u,v')$. Further, from $\eta'' \leq \eta$ we obtain that $i_q^\ast(\eta - \eta'')$ is a positive linear functional such that Proposition \ref{prop.preccurly.positive} implies $i_q^\ast \eta \preccurlyeq i_q^\ast \eta''$. 
    \item[(iii)] Let $k\in K_v^\downarrow$, i.e. $k \preccurlyeq K_v$. Then from (K2) we have $K_v \preccurlyeq K_v \preccurlyeq K_{v'}$ and Proposition \ref{prop.transitive} gives $k \preccurlyeq K_{v'}$, i.e. $k \in K_{v'}^\downarrow$. That is, we even have $K_v^\downarrow \subset K_{v'}^\downarrow$.
\end{itemize}

\textit{Second}, $\ul K(u,v) \subset \ul K(u,v')$ holds trivially since $\underline K(u,v) \equiv W$.

\textit{Third}, for $u\in \ul S(v)$ we have to prove that $\ul T(u,v') \subset \ul T(u,v)$, where $\ul T(u,v) = u\wedge \mathcal D(K_v)$. To this end, let $k \in K_v^\downarrow$ be such that $u \in \mathcal D(k)$, let $w' \in \mathcal D(K_{v'})$ and let $w \in \mathcal D(K_v)$ be arbitrary. Then $u\wedge w' = u\wedge((u\vee w)\wedge w') \in u \wedge \mathcal D(K_v)$ since $\mathcal D(k) \preccurlyeq \mathcal D(K_v) \preccurlyeq \mathcal D(K_{v'})$, see Proposition \ref{prop.transitive}. 

\textit{Fourth}, $\ul T(u,v') \leq^\ast u$ holds true by definition of $\wedge$.
\end{proof}

It now remains to prove that all values of $\ul S$ are directed upward, and that it holds $S(v) \subset \ul S(v) \leq^\ast S(v)$. Of these properties, the inclusion $S(v) \subset \ul S(v)$ is easy to show due to the variational definition of $\ul S$:

\begin{proposition}
For all $v\in D$ it holds $S(v) \subset \ul S(v)$.
\end{proposition}

\begin{proof}
We once again apply Lemma \ref{lem.dependence}. To this end, let $v\in D$ and $u\in S(v)$. Then from $K_v \preccurlyeq^\ast K_v^\downarrow$ it follows $A(u,v) \preccurlyeq^\ast \ul A(u,v)$. Further, $\ul T(u,v) \subset W$ as well as $\ul T(u,v) \leq^\ast u$ hold trivially.
\end{proof}

The remaining properties rely on sophisticated truncation arguments which allow us to apply the following existence theorem for perturbed coercive pseudomonotone multifunctions:

\begin{theorem}
\label{theorem.surjective.psm}
Let $X$ be a real reflexive Banach space, let $M\colon X\to \mathcal P(X)$ be a maximal monotone mapping, suppose $M(u_0)\neq\emptyset$, and let $T\colon X\to \mathcal P(X^\ast)$ be a bounded, pseudomonotone mapping having only non-empty, closed and convex values. If $T$ is \textbf{coercive} with respect to $u_0$, i.e. there exists a real-valued function $c$ on $\mathbb R^+$ with $\lim_{s\to\infty} c(s) = \infty$ such that for all $u\in X$, $u^\ast\in T(u)$ one has $\langle u^\ast,u-u_0\rangle \geq c(\norm{u})\norm{u}$,
then $M+T$ is {surjective}, i.e. for all $y\in X^\ast$ there is $x\in X$ such that $y\in (M+T)(x)$.
\end{theorem}

The proof of Theorem \ref{theorem.surjective.psm} together with further information is contained in, e.g., \cite{NaPa95}.
We will skip the rather technical definition of multi-valued pseudomonotonicity (but note that each single-valued pseudomonotone mapping is also pseudomonotone as a multi-valued mapping, and that sums of pseudomonotone mappings are pseudomonotone, too) and invoke \cite[Proposition 2.2]{NaPa95} to formulate the following useful result:

\begin{proposition}
\label{prop.psm}
Let $X$ be a real reflexive Banach space and let $T\colon X\to \mathcal P(X^\ast)$ be a multi-valued mapping satisfying the following properties:
\begin{itemize}
\item[(i)] All values of $T$ are non-empty, closed and convex,
\item[(ii)] $T$ is \textbf{bounded}, i.e. $T$ maps bounded sets to bounded ones,
\item[(iii)] the graph $Gr(T)$ of $T$ is \textbf{sequentially weakly closed}, i.e. for all weakly convergent sequences $u_n\rightharpoonup u$ in $X$ and $u_n^\ast\rightharpoonup u^\ast$ in $X^\ast$ with $u_n^\ast\in T(u_n)$ one has $u^\ast\in T(u)$,
\item[(iv)] the duality pairing is \textbf{($\bs{w\times w}$)-continuous} on $Gr(T)$, i.e. $u_n\rightharpoonup u$ in $X$ and $u_n^\ast\rightharpoonup u^\ast$ in $X^\ast$ with $u_n^\ast\in T(u_n)$ imply $\langle u_n^\ast,u_n\rangle \to \langle u^\ast,u\rangle$.
\end{itemize}
Then $T$ is pseudomonotone.
\end{proposition}

Thanks to Proposition \ref{prop.psm} we are able to show that upper Carath\'eodory functions induce a pseudomonotone operator, which we will need to prove the existence of solutions. The proof of the next lemma is contained in \cite{CaTi18}, thus we keep it short. Note, however, that we do not need any knowledge about Hausdorff upper semicontinuous multifunctions as in \cite[Lemma 2.5]{Le14} which allows for a rather elementary approach.

\begin{lemma}
\label{lem.G.pseudomonotone}
Let $g\colon \Omega \times \mathbb R \to \mathcal P(\mathbb R)$ be a multifunction whose values are non-empty, closed and convex. Suppose further that $g$ is \textbf{upper Carath\'eodory}, i.e. $x\mapsto g(x,s)$ is measurable for all $s\in \mathbb R$ and $s\mapsto g(x,s)$ is upper semicontinuous for a.e. $x\in\Omega$, and let there be some $b \in L$ such that $g$ satisfies, for a.e. $x\in\Omega$,
\begin{equation}
\label{eq.growth.g}
|y| \leq b(x)\quad\text{for all }y\in g(x,s)\text{, a.e }x\in \Omega\text{ and all }s\in\mathbb R.
%\sup\{|y|: y\in g(x,s), s\in \mathbb R\} \leq b(x).
\end{equation}
Then the selection mapping
\begin{equation*}
G\colon W \to \mathcal P(W^\ast),\quad G(u) := \{i_q^\ast \eta: \eta \textrm{ is a measurable selection of } g(\cdot,u)\}
\end{equation*}
is well-defined, pseudomonotone, bounded and has only non-empty, closed and convex values.
\end{lemma}

\begin{proof}
Let $u\in W$ be given. Due to the growth condition \eqref{eq.growth.g} we see that the measurable selections of $g(\cdot,u)$ are uniformly bounded in $L$ independent of $u$. Since %$i_q\colon W\hookrightarrow L'$ and its dual 
$i_q^\ast\colon L \hookrightarrow W^\ast$ is a compact embedding, we conclude that $G(u)$ is well-defined and uniformly bounded in $W^\ast$, again independent of $u$.
Further, $G(u)$ is convex and closed in $W^\ast$ since $g$ has convex and closed values. Indeed, let $(\eta_n)\subset g(\cdot,u)$ be some sequence such that the sequence $(i_q^\ast \eta_n)$ converges in $W^\ast$. Since $g(\cdot,u)$ is bounded in the reflexive space $L$, we conclude, up to a subsequence, $\eta_n\rightharpoonup \eta$ in $L$ for some $\eta$, and due to Proposition \ref{prop.F.closed} we have $\eta\subset g(\cdot,u)$. Because $i_q^\ast$ is linear and bounded, it follows $i_q^\ast \eta_n \rightharpoonup i_q^\ast \eta$ in $W^\ast$ and we infer that $G(u)$ is closed.

In order to apply Proposition \ref{prop.psm}, it remains to verify that $Gr(G)$ is sequentially weakly closed and that the duality pairing is (w$\times$w)-continuous on $Gr(G)$. 
To this end, let $(u_n)\subset W$ and $(i^\ast \eta_n)\subset W^\ast$ be sequences such that $\eta_n\subset g(\cdot,u_n)$ for all $n$, and, for some $u\in W$, $\eta^\ast\in W^\ast$, $u_n\rightharpoonup u$ in $W$ and $i_q^\ast\eta_n\rightharpoonup \eta^\ast$ in $W^\ast$. By the compact embedding $i_q$ and since $(\eta_n)$ is bounded in $L$, we have, up to a subsequence, $u_n\to u$ in $L^\ast$, $u_n\to u$ a.e. and $\eta_n\rightharpoonup \eta$ in $L$ for some $\eta\in L$. We apply once more Proposition \ref{prop.F.closed} to conclude $\eta\subset g(\cdot,u)$. Since $i_q^\ast$ is linear and bounded, we infer $i_q^\ast\eta_n\rightharpoonup i_q^\ast \eta$, and thus $\eta^\ast = i_q^\ast\eta\in G(u)$. Further, we apply the identity $\langle i_q^\ast \theta,v\rangle = \langle \theta,v\rangle$ (for $\theta \in L$, $v\in W$) and the triangle inequality to conclude
\begin{equation*}
%\begin{split}
0\leq |\langle \eta_n^\ast,u_n\rangle - \langle \eta^\ast,u\rangle| 
%&\leq |\langle \eta_n,u_n-u\rangle_{q}| + |\langle \eta_n-\eta,u\rangle_{q}| \\
\leq \norm{\eta_n}_{L}\norm{u_n-u}_{L^\ast} + |\langle \eta_n-\eta,u\rangle|\to 0.
%\end{split}
\end{equation*}
Thus, $G$ fulfills the requirements of Proposition \ref{prop.psm} and is thus pseudomonotone.
\end{proof}

Note that \eqref{eq.growth.g} is a global growth condition while (F4) is only a local growth condition. This is why we need some truncation procedure in the proof of the next theorem.

\begin{theorem}
\label{theorem.aux.existence}
Let $v\in D$ %[\ul u,\ol u]_V$
be arbitrary, and let $\ul v_i \in \ul S(v)$ and $\ol v_i\in \ol S(v)$, $i=1,2$, be %sub-supersolutions 
such that 
\begin{equation*}
\ul u \leq \ul v\defeq \ul v_1 \vee \ul v_2 \leq \ol v_1 \wedge \ol v_2 \eqdef \ol v \leq \ol u.
\end{equation*}
Then there is %a solution 
$u\in S(v)$ such that $\ul v\leq u \leq \ol v$. Especially, $S(v)$ is directed upward and it holds $\ul S(v) \leq^\ast S(v)$.
\end{theorem}

\begin{proof}
\textbf{Step 1: Auxiliary functions}~~~~~ We are going to define an auxiliary problem whose solutions will be the desired elements of $S(v)$. To this end, recall that, by definition, for $i=1,2$ there are $\ul \eta_i \in F(\ul v_i,v)$ and $\ul k_i \in \Gamma$ as well as $\ol \eta_i \in F(\ol v_i,v)$ and $\ol k_i \in \Gamma$ such that we have the relations
\begin{equation*}
\ul v_i \in \mathcal D(\ul k_i), \quad \ul k_i \preccurlyeq K_v \preccurlyeq \ol k_i, \quad\mathcal D(\ol k_i) \ni \ol v_i
\end{equation*}
and such that the following inequalities hold:
\begin{align}
\langle E\ul v_i, w-\ul v_i\rangle + \int_\Omega \ul \eta_i(w-\ul v_i) + \ul k_i(w) - \ul k_i(\ul v_i) \geq 0,\quad \text{for all }w\in \ul v_i\wedge \mathcal D(K_v),  \label{eq.ulv}\\
\langle E\ol v_i, w-\ol v_i\rangle + \int_\Omega \ol \eta_i(w-\ol v_i) + \ol k_i(w) - \ol k_i(\ol v_i) \geq 0,\quad \text{for all }w\in \ol v_i\wedge \mathcal D(K_v). \label{eq.olv}
\end{align}
Define the functions $\ul \eta, \ol \eta\colon \Omega \to \mathbb R$ pointwise a.e. by
\begin{equation*}
\ul\eta(x) := \begin{cases}
\ul\eta_1(x) &\text{ if }\ul v_1(x) \geq \ul v_2(x),\\ 
\ul\eta_2(x) &\text{ if }\ul v_1(x) < \ul v_2(x), 
\end{cases}\qquad
\ol\eta(x): = \begin{cases}
\ol\eta_1(x) &\text{ if }\ol v_1(x) \leq \ol v_2(x),\\ 
\ol\eta_2(x) &\text{ if }\ol v_1(x) > \ol v_2(x), 
\end{cases}
\end{equation*}
and note that $\ul \eta \in F(\ul v,v)$ and $\ol \eta \in F(\ol v,v)$. Further, we need three auxiliary functions $d$, $g$ and $h$, which will be introduced subsequently. 

\textit{First}, let us introduce the cut-off function
\begin{equation*}
d\colon\Omega \times \mathbb R\to\mathbb R,\quad d(x,s) := \begin{cases}
-(\ul v(x)-s)^{p-1}&\text{ if } s <\ul v(x),\\
0                  &\text{ if } \ul v(x) \leq s \leq \ol v(x),\\
(s-\ol v(x))^{p-1} &\text{ if } \ol v(x) < s.
\end{cases}
\end{equation*}
Obviously, $d$ is a Carath\'eodory function and satisfies the growth condition
\begin{equation}
\label{est.d.above}
|d(x,s)|\leq d_0\left(|\ul v(x)|^{p-1}+|s|^{p-1}+|\ol v(x)|^{p-1} \right)
\end{equation}
for some constant $d_0>0$.
Hence, the Nemytskij operator $v\mapsto d(\cdot,v)$ is known to be continuous and bounded from $L^p(\Omega)$ to its dual space. Thanks to the compact embedding $W\hookrightarrow L^p(\Omega)$ we conclude that the mapping $D\colon W\to W^\ast$, defined by 
\begin{equation*}
\langle Dv,w\rangle := \int_\Omega d(\cdot,v)w\quad\text{for all }w\in W,
\end{equation*}
is bounded and completely continuous, thus pseudomonotone. For further use, let us note that there are constants $d_1, d_2 > 0$ such that both $-(t-s)^{p-1}s$ for $s\leq t$ and $(s-t)^{p-1}s$ for $s\geq t$ are bounded from below by $d_1 |s|^p - d_2 |t|^{p-1}|s|$. One may take $d_1 = 1$ and $d_2 = 2^{2-p}$ if $p \leq 2$ and $d_1 = 2^{2-p}$ and $d_2 = 1$ if $p\geq 2$. By these estimates and by use of Young's inequality with epsilon, we obtain for all $v\in W$
\begin{equation}
\label{est.d.below}
\langle Dv,v\rangle 
%\geq \int_{u\leq \ul v} d_0 |v|^p - d_1 |\ul v|^{p-1}|v| + \int_{u\geq \ol v} d_0 |v|^p - d_1 |\ol v|^{p-1}|v| 
\geq d_3 \|v\|_p^p -d_3(\norm{\ul v}_p^p + \norm{\ol v}_p^p),
\end{equation}
where $d_3 > 0$ is some constant not depending on $v$.

\textit{Second}, let us truncate $f$ to define the multifunction
\begin{equation*}
%\label{eq.def.g}
g\colon \Omega\times\mathbb R\to \mathcal P(\mathbb R),\quad g(x,s) := \begin{cases}
\{\ul \eta(x)\,&\text{ if } s <\ul v(x),\\
f(x,s,v(x))                  &\text{ if } \ul v(x) \leq s \leq \ol v(x),\\
\{\ol\eta(x)\} &\text{ if } \ol v(x) < s.
\end{cases}
\end{equation*}
From (F1) and (F2) we have that $(x,s)\mapsto f(x,s,v(x))$ is upper Carath\'eodory and it is readily seen (by invoking $\ul \eta \subset f(\cdot,\ul v,v)$ and $\ol \eta \subset f(\cdot,\ol v,v)$) that $g$ is upper Carath\'eodory, too. Since $g$ has only non-empty, closed and convex values, and due to (F4) and Lemma \ref{lem.G.pseudomonotone}, the mapping 
\begin{equation*}
G\colon W \to \mathcal P(W^\ast),\quad G(u) := \{i_q^\ast \eta: \eta \textrm{ is a measurable selection of } g(\cdot,u)\}
\end{equation*}
is well-defined, pseudomonotone, bounded and its values are non-empty, closed and convex.

\textit{Third}, in order to define $h$, let us introduce the following notation: For any real numbers $x_1 < x_2$ and $y_1, y_2$, denote by
\begin{equation*}
l=\big[(x_1,y_1) \rightsquigarrow (x_2,y_2)\big]
\end{equation*}
the continuous piecewise linear function $l\colon \mathbb R \to \mathbb R$ that satisfies $l(x) = y_1$ for $x\leq x_1$, $l'(x) = (y_2-y_2)/(x_2-x_1)$ for $x_1 < x < x_2$, and $l(x) = y_2$ for $x\geq x_2$. 

This said, for $i=1,2$ we introduce the functions
\begin{align*}
\ul \theta_i \colon \Omega \times \mathbb R \to \mathbb R,\quad \ul \theta_i(x,\cdot) := \big[\big(\ul v_i(x),\ul \eta(x) - \ul \eta_i(x)\big) \rightsquigarrow \big(\ul v(x),0\big)\big],\\
\ol \theta_i \colon \Omega \times \mathbb R \to \mathbb R,\quad \ol \theta_i(x,\cdot) := \big[\big(\ol v(x),0\big) \rightsquigarrow \big(\ol v_i(x),\ol \eta_i(x) -\ol \eta(x)\big)\big].
\end{align*}
It is easy to check that $\ul \theta_i$ and $\ol \theta_i$ are measurable in the first argument and, of course, continuous in the second. Let us combine them to form the Carath\'eodory function
\begin{equation*}
h\colon \Omega\times \mathbb R\to\mathbb R,\quad h(x,s) := |\ul\theta_1(x,s)| + |\ul\theta_2(x,s)| - |\ol\theta_1(x,s)| - |\ol\theta_2(x,s)|.
\end{equation*}

We have constructed $h$ in such a way that, for all $u\in L^0(\Omega)$, $i=1,2$, the inequalities 
\begin{equation}
\label{ineq.theta}
\ul\eta - \ul\eta_i - h(\cdot,u) \leq 0\quad \text{on }\{u<\ul v_i\},\qquad \ol\eta_i - \ol\eta + h(\cdot,u) \leq 0\quad \text{on }\{\ol v_i < u\}
\end{equation}
hold true. Indeed, if, e.g., $v(x) < \ul v_i(x) \leq \ul v(x)$ for some $x\in\Omega$, then we obtain per definition $\ul\theta_i(x,v(x)) = \ul \eta(x) - \ul \eta_i(x)$ and $\ol \theta_i(x,v(x)) = 0$, $i=1,2$, which implies the first inequality in (\ref{ineq.theta}). Moreover, $h(x,s) = 0$ if $\ul v(x) \leq s \leq \ol v(x)$.

Further, $h(x,\cdot)$ is obviously bounded by some function $\eta\in L$, so that the Nemytskij operator $v\mapsto h(\cdot, v)$ is known to be a continuous and bounded mapping from $L$ to $L^\ast$. Thanks to the compact embedding $W\hookrightarrow L$ we conclude that the composed mapping $H\colon W\to W^\ast$, defined by
\begin{equation*}
\langle Hv,w\rangle :=\int_\Omega h(\cdot,v)w\quad\text{for all }w\in W,
\end{equation*}
is bounded and completely continuous, thus pseudomonotone.\\

\noindent\textbf{Step 2: Solutions of auxiliary problem}~~~~~ Let us consider the multi-valued mapping
\begin{equation*}
E+D+G-H+\partial K_v\colon W \to \mathcal P(W^\ast).
\end{equation*}
It is the sum of the bounded, pseudomonotone single-valued mappings $E$, $D$ and $-H$, the bounded, pseudomonotone mapping $G$ with non-empty, closed and convex values, and the maximal monotone mapping $\partial K_v$. Further, $E$ is coercive with respect to any $u_0\in \mathcal D(\partial K_v)$, since by (E2)---(E3) we have for all $u,u_0\in W$
\begin{equation}
\label{eq.coer.1}
\langle Eu, u-u_0\rangle 
\geq \alpha_2\norm{\nabla u}_p^p - \norm{a_2}_1 - (\alpha_3 \norm{\nabla u}_p^{p-1} + \norm{a_3}_{p'})\norm{\nabla u_0}_p,
\end{equation}
which implies
\begin{equation*}
    \frac{\langle Eu,u-u_0\rangle}{\|u\|_W} \to \infty\quad\text{as}\quad \|u\|_W \to \infty
\end{equation*}
(note again that $\|\nabla u\|_p$ defines an equivalent norm on $W$).
%The other summands do not affect the coercivity. Indeed,
Further, for all $u,u_0\in W$ and any $\eta\subset g(\cdot,u)$ we have, since the selections of $g(\cdot,u)$ and $Hu$ are uniformly bounded,
\begin{equation}
\label{eq.coer.2}
\langle i^\ast\eta - Hu,u-u_0\rangle \geq -\norm{\eta-Hu}_L\norm{u-u_0}_{L'} \geq -c_0 \norm{u}_W -c_1
\end{equation}
for some constants $c_0,c_1 > 0$. In addition, estimates (\ref{est.d.above}) and (\ref{est.d.below}) imply 
\begin{equation}
\label{eq.coer.3}
\langle Du,u-u_0\rangle\geq - d_3 (\norm{\ul v}_p^p + \norm{\ol v}_p^p) - d_0(\norm{\ul v}_p^{p-1}+\norm{u}_p^{p-1} + \norm{\ol v}_p^{p-1}) \geq d_4 - d_5 \norm{u}_W^{p-1}
\end{equation}
for constants $d_4, d_5 > 0$. From \eqref{eq.coer.1}, \eqref{eq.coer.2} and \eqref{eq.coer.3} it follows that $E+D+G-H\colon W \to \mathcal P(W^\ast)$ is coercive with respect to $u_0$.

Taken together, all conditions of Theorem \ref{theorem.surjective.psm} are fulfilled, and thus there is some $u\in W$ such that $(E+D+G-H+\partial K_v)(u) \ni 0$. By definition of the subdifferential this implies $u\in \mathcal D(K_v)$ and that there is $\eta\subset g(\cdot,u)$ such that, for all $w\in W$,
\begin{equation}
\label{eq.aux.1}
%\begin{split}
\langle Eu,w-u\rangle +  \int_\Omega d(\cdot,u)(w-u) + \int_\Omega \left(\eta - h(\cdot,u)\right) (w-u) + K_v(w) - K_v(u) \geq 0. %\\
%\quad\text{for all }w\in \mathcal D(K_v).
%\end{split}
\end{equation}
\\
\textbf{Step 3: Solution of the QVIP} ~~~~~ In this last step, let us check that for any $u\in W$ and $\eta \subset g(\cdot,u)$ with (\ref{eq.aux.1}) we have $\ul v \leq u \leq \ol v$. To this end, for $i=1,2$, take 
\begin{equation*}
w = \ul v_i - (\ul v_i - u)^+ = \ul v_i \wedge u \in \ul v_i \wedge \mathcal D(K_v)
\end{equation*}
as test function in (\ref{eq.ulv}), take 
\begin{equation*}
w = u + (\ul v_i - u)^+ = \ul v_i\vee u \in \mathcal D(\ul k_i) \vee \mathcal D(K_v) \subset  \mathcal D(K_v)
\end{equation*}
as test function in (\ref{eq.aux.1}), and add the resulting inequalities to obtain
\begin{equation}
\label{eq.usolution.1}
\begin{split}
\langle Eu-E\ul v_i,(\ul v_i - u)^+\rangle + \int_\Omega d(\cdot,u)(\ul v_i - u)^+ + \int_\Omega \left(\eta -\ul \eta_i - h(\cdot,u)\right) (\ul v_i - u)^+ &\\
+ K_v(\ul v_i \vee u) - K_v(u) + \ul k_i(\ul v_i \wedge u) - \ul k_i(\ul v_i) &\geq 0. 
\end{split}
\end{equation}
By (E4) and the identity $\nabla (\ul v_i - u)^+ = \chi_{\{\ul v_i \geq u\}}\nabla (\ul v_i - u)$ (where $\chi_M$ denotes the indicator function of a set $M$) we deduce
\begin{equation}
\label{eq.usolution.2}
\langle Eu-E\ul v_i,(\ul v_i - u)^+\rangle \leq 0.
\end{equation}
Further, it follows from \eqref{ineq.theta}
\begin{equation}
\label{eq.usolution.3}
\int_\Omega \left(\eta -\ul \eta_i - h(\cdot,u)\right) (\ul v_i - u)^+ \leq 0,
\end{equation}
and $\ul v_i \in \mathcal D(\ul k_i)$, $u \in \mathcal D(K_v)$ and $\ul k_i \preccurlyeq K_v$ imply
\begin{equation}
\label{eq.usolution.4}
 K_v(\ul v_i \vee u) - K_v(u) + \ul k_i(\ul v_i \wedge u) - \ul k_i(\ul v_i) \leq 0.
\end{equation}
Combining \eqref{eq.usolution.1}---\eqref{eq.usolution.4}, we deduce
\begin{equation}
\label{eq.usolution.5}
\int_\Omega d(\cdot,u)(\ul v_i - u)^+ \geq 0.
\end{equation}
By definition of $d$, (\ref{eq.usolution.5}) implies $\norm{(\ul v_i - u)^+}_p^p \leq 0$, which in turn implies $\ul v_i \leq u$, $i= 1,2$. Since $\ul v =\ul v_1 \vee \ul v_2$, we have $\ul v \leq u$. 

In a similar way we conclude $u\leq \ol v$ by taking $w= \ol v_i + (u-\ol v_i)^+ = \ol v_i \vee u$ as test function in (\ref{eq.olv}) and $w=u - (u-\ol v_i)^+ = u\wedge \ol v_i$ as test function in (\ref{eq.aux.1}).

In view of $\ul v \leq u \leq \ol v$, (\ref{eq.aux.1}) reduces to
\begin{equation*}
\langle Eu,w-u\rangle + \int_\Omega \eta (w-u) + K_v(w) - K_v(u) \geq 0\quad\text{for all }w\in W, %\mathcal D(K_v),
\end{equation*}
and since $\eta \subset g(\cdot,u) = f(\cdot,u,v)$, we have $u\in S(v)$.
\end{proof}

Finally, existence of greatest solutions follows from Theorem \ref{thm.greatest.fixedpoint}, and since all results remain true if we consider the dual order $\geq$, we obtain also a smallest solution within the order interval $[\ul u,\ol u]$:

\begin{theorem}
Suppose (S), (E1)---(E4), (F1)---(F5) and (K1)---(K3). Then Problem \eqref{eq.problem} has both a smallest and a greatest solution in $[\underline u,\overline u]$.
\end{theorem}

%%%%%%% bibliography

\printbibliography

\end{document}